\documentclass{article}[12pt]

\usepackage{graphicx} 
\usepackage{amssymb}
\usepackage{amsthm}
\usepackage{bbm}
\usepackage{enumerate}
\usepackage[clock]{ifsym}
\usepackage{mathtools}
\usepackage{mathrsfs}
\usepackage[english]{babel}
\usepackage{mathrsfs}
\usepackage{fullpage}
\usepackage{xcolor}
\usepackage{hyperref}
\usepackage{tikz}
\usetikzlibrary{decorations.markings, arrows.meta}
\usepackage{pgfplots}
\pgfplotsset{compat=newest}
\usepgfplotslibrary{dateplot}
\usepgfplotslibrary{colormaps}

\newtheorem{thm}{Theorem}[section]
\newtheorem{lm}[thm]{Lemma}
\newtheorem{defn}[thm]{Definition}
\newtheorem{prop}[thm]{Proposition}

\numberwithin{equation}{section}
\newcommand{\Ai}{\mathrm{Ai}}
\newcommand{\Airy}{\mathcal{A}}

\newcommand{\bfK}{\mathbf{K}}
\newcommand{\bfL}{\mathbf{L}}

\newcommand{\bsn}{{\boldsymbol{n}}}

\newcommand{\bsz}{{\boldsymbol{z}}}

\newcommand{\complexC}{\mathbb{C}}

\newcommand{\FGUE}{F_{\mathrm{GUE}}}

\newcommand{\hKPZ}{\mathcal{H}^{\mathrm{KPZ}}}
\newcommand{\identity}{\mathbf{1}}

\newcommand{\inn}{\mathrm{in}}

\newcommand{\KAiry}{K_{\mathrm{Ai}}^{\mathrm{ext}}}

\newcommand{\out}{\mathrm{out}}
\newcommand{\prob}{\mathbb{P}}
\newcommand{\rC}{\mathrm{C}}
\newcommand{\rd}{\mathrm{d}}
\newcommand{\realR}{\mathbb{R}}

\newcommand{\rI}{\mathrm{I}}

\newcommand{\rmi}{\mathrm{i}}

\newcommand{\rL}{\mathrm{L}}

\newcommand{\rR}{\mathrm{R}}
\newcommand{\sgn}{\operatorname{sgn}}

\renewcommand{\Re}{\mathrm{Re}}
\renewcommand{\Im}{\mathrm{Im}}

\title{On the multipoint distribution formulas of the parabolic Airy process}
\author{Zhipeng Liu\footnote{Department of Mathematics, University of Kansas, Lawrence, KS 66045. Email: \texttt{zhipeng@ku.edu}} \and Aaron Ortiz\footnote{Department of Mathematics, University of Kansas, Lawrence, KS 66045. Email: \texttt{aortiz@ku.edu}}}

\begin{document}

\maketitle

\begin{abstract}
    The parabolic Airy process is the Airy$_2$ process minus a parabola, initially defined by its finite-dimensional distributions, which are given by a Fredholm determinant formula with the extended Airy kernel. This process is also the one-time spatial marginal of the KPZ fixed point with the narrow wedge initial condition. There are two formulas for the space-time multipoint distribution of the KPZ fixed point with the narrow wedge initial condition obtained by \cite{Johansson-Rahman21} and \cite{Liu22}. Especially, the equal-time case of \cite{Liu22} gives a different formula of the multipoint distribution of the parabolic Airy process. In this paper, we present a direct proof that this formula matches the one with the extended Airy kernel. Some byproducts in the proof include several new formulas for the parabolic Airy process, and a generalization of the Andreief's identity.
\end{abstract}

\section{Introduction}
The parabolic Airy process is defined to be $\Airy(\alpha)=\Airy_2(\alpha)-\alpha^2$, $\alpha\in\realR$, where $\Airy_2(\alpha)$ is the Airy$_2$ process introduced by Pr\"ahofer and Spohn  \cite{Prahofer-Spohn02}. It is conjectured to be a universal limit of the models in the (1+1)-dimensional Kardar-Parisi-Zhang universality class with the narrow-wedge initial condition \cite{Baik-Deift-Johansson99,Johansson00,Ferrari08b,Tracy-Widom09,Amir-Corwin-Quastel11,Borodin-Corwin13,Borodin-Corwin-Gorin16,Ferrari_2023}. The parabolic Airy process can be defined by its finite-dimensional distributions
\begin{equation}
\label{eq:def_AiryProcess}
\prob\left(\bigcap_{i=1}^m \left\{ \Airy (\alpha_i) \le \beta_i \right\}\right) =
\prob\left(\bigcap_{i=1}^m \left\{ \Airy_2 (\alpha_i) \le \beta_i+\alpha^2_i \right\}\right)
= \det\left(\rI -\chi^{1/2} \KAiry \chi^{1/2}\right)_{L^2(\{\alpha_1,\ldots,\alpha_m\}\times\realR)},
\end{equation}
where $\alpha_1<\cdots<\alpha_m$, $\chi$ is the indicator function defined by
\begin{equation}
\chi(\alpha_i,x) = \identity_{(\beta_i+\alpha^2_i,\infty)}(x) =\begin{dcases}
    1, & x> \beta_i+\alpha^2_i,\\
    0, & \text{elsewhere},
\end{dcases}
\end{equation}
and $\KAiry$ is the extended Airy kernel defined by
\begin{equation}
\KAiry(\alpha_i,x;\alpha_j,y) 
= \begin{dcases}
    \int_0^\infty e^{-z(\alpha_i-\alpha_j)} \Ai(x+z)\Ai(y+z)\rd z, & \text{ if }\alpha_i\ge \alpha_j,\\
    -\int_{-\infty}^0 e^{-z(\alpha_i-\alpha_j)} \Ai(x+z)\Ai(y+z)\rd z, & \text{ if }\alpha_i< \alpha_j.
  \end{dcases}
\end{equation}
Here $\Ai(x)$ denotes the Airy function
\begin{equation}
\label{eq:def_Airy}
\Ai(x) = \int_{\Gamma_{\rL}} e^{-\frac{1}{3}u^3 +xu} \frac{\rd u}{2\pi\rmi},
\end{equation}
where $\Gamma_\rL$ is a contour on the complex plane that goes from $\infty e^{-2\pi\rmi/3}$ to $\infty e^{2\pi\rmi/3}$.

When $m=1$, denote $\alpha_1=\alpha$ and $\beta_1=\beta$. The one-point distribution of the parabolic Airy process is given by 
\begin{equation}
\label{eq:GUE_marginal}
\prob\left(\Airy(\alpha)\le \beta\right) = \FGUE\left(\beta+\alpha^2\right),\quad \alpha,\beta\in\realR,
\end{equation}
where $\FGUE$ is the GUE Tracy-Widom distribution.

\bigskip

The parabolic Airy process can be viewed as a special case of the one-time marginals of the KPZ fixed point, which is a space time random field constructed by \cite{Matetski-Quastel-Remenik21}. Similarly to the parabolic Airy process, the KPZ fixed point is also conjectured and partially proved to be a universal limit of  the models in the Kardar-Parisi-Zhang universality class \cite{Matetski-Quastel-Remenik21,Dauvergne-Ortmann-Virag22,Quastel-Sarkar23,Virag20,Aggarwal-Corwin-Hegde24,Dauvergne-Zhang25}. The KPZ fixed point depends on the initial condition. Throughout this paper, we will only consider one special initial condition where the parabolic Airy process arises. Denote $\hKPZ (\alpha,\tau)$, with $(\alpha,\tau)\in\realR\times[0,\infty)$, the KPZ fixed point with the narrow-wedge initial condition $\hKPZ(\alpha,0) = -\infty \cdot \identity_{\realR\setminus\{0\}}(\alpha)$, i.e., $\hKPZ(\alpha,0)=0$ when $\alpha=0$, and $-\infty$ elsewhere. 

The KPZ fixed point $\hKPZ(\alpha,\tau)$ has the following well known $1:2:3$ scaling invariance
\begin{equation}
\hKPZ(\alpha,\tau) \stackrel{\mathrm{d}}{=} \epsilon^{-1/3} \hKPZ(\epsilon^{2/3}\alpha,\epsilon \tau),
\end{equation}
where $\stackrel{\mathrm{d}}{=}$ denotes the equation in distribution. Moreover, its one-time marginals are given by the (rescaled) parabolic Airy process
\begin{equation}
\hKPZ(\alpha,\tau_0) \stackrel{\mathrm{d}}{=} \tau_0^{1/3}\Airy \left(\tau_0^{-2/3}\alpha\right) 
\end{equation}
for any fixed $\tau_0>0$. Especially, for $\tau_{0} = 1$, we can write
\begin{equation}
\label{eq:relation_Airy2KPZ}
\Airy(\alpha) \stackrel{\mathrm{d}}{=} \hKPZ(\alpha,1).
\end{equation}

On the other hand, the exact formulas of the finite-dimensional distributions of the KPZ fixed point (with the narrow-wedge initial condition)  $\hKPZ(x,t)$ for general space time points were also obtained in \cite{Johansson-Rahman21,Liu22}. The formulas in these two papers are different, with both being very complicated. The formula in \cite{Johansson-Rahman21} is valid for points with different time parameters (and hence one still needs to use the continuity of the KPZ fixed point and take an extra limit to get a formula for the case when the time parameters are equal), while the formula in \cite{Liu22} holds for arbitrary space time points, which could include some equal time parameters. It implies that the equal-time multipoint distribution formula in \cite{Liu22} also gives the finite-dimensional distributions of the parabolic Airy process. However, a direct verification was missing. The motivation of this paper is to give a direct proof that the equal-time formula of \cite{Liu22} indeed matches the original multipoint distribution formula  \eqref{eq:def_AiryProcess} of the parabolic Airy process.

\bigskip

Let us introduce the formula of \cite{Liu22} below. Define an order $\prec$ in $\realR\times\realR_+$ as follows, here $\realR_+$ denotes the set of positive real numbers. $(\alpha,\tau)\prec (\alpha',\tau')$ if and only if one of the following two conditions are satisfied:
\begin{enumerate}[(1)]
\item $\tau<\tau'$, or
\item $\tau=\tau'$ and $\alpha<\alpha'$.
\end{enumerate}

\begin{thm}[\cite{Liu22}]
\label{thm:Liu}
Assume the points $(\alpha_i,\tau_i)\in\realR\times\realR_+$ are ordered $(\alpha_1,\tau_1)\prec  \cdots\prec(\alpha_m,\tau_m)$. We have
\begin{equation}
\label{eq:formula_Liu}
     \prob\left(  \bigcap_{i =1}^ m \left\{ \hKPZ (\alpha_i,\tau_i) \le  \beta_i \right\}\right)
     = \oint_0 \cdots \oint_0 \mathcal{D}_{(\alpha_1,\tau_1),\ldots,(\alpha_m,\tau_m)}(\bsz; \beta_1,\ldots, \beta_m) \prod_{i=1}^{m-1} \frac{\rd z_i}{2\pi\rmi z_i(1-z_i)},
\end{equation}
where $\bsz=(z_1,\cdots,z_{m-1})$, and $\mathcal{D}_{(\alpha_1,\tau_1),\ldots,(\alpha_m,\tau_m)}(\bsz; \beta_1,\ldots, \beta_m)$ is a function defined in Definition \ref{def:calD}. Moreover, the symbol $\oint_0$ denotes the integral along a small circle around the origin, with the counterclockwise orientation.
\end{thm}

As a special case, when $\tau_1=\cdots=\tau_m=1$ and $\alpha_1<\cdots<\alpha_m$, the formula \eqref{eq:formula_Liu} is the same as $\prob\left(\bigcap_{i =1}^ m \left\{ \Airy(\alpha_i) \le  \beta_i \right\}\right)$ by the relation \eqref{eq:relation_Airy2KPZ} between the parabolic Airy process and the KPZ fixed point. Thus, we have the following formula for the parabolic Airy process
\begin{equation}
\label{eq:Airy_Liu}
\prob\left(\bigcap_{i=1}^m \left\{ \Airy (\alpha_i) \le \beta_i \right\}\right) = \oint_0\cdots\oint_0 \mathcal{D}_{(\alpha_1,1),\ldots,(\alpha_m,1)}(\bsz;\beta_1 ,\ldots,\beta_m ) \prod_{i=1}^{m-1} \frac{\rd z_i}{2\pi\rmi z_i(1-z_i)}
\end{equation}
for $\alpha_1<\cdots<\alpha_m$.

It turns out that \eqref{eq:Airy_Liu} can be simplified to a Fredholm determinant with kernel defined on a contour on a complex plane, see Proposition \ref{prop:alt_formula}. We will also show that the simplified Fredholm determinant formula matches the original formula \eqref{eq:def_AiryProcess} in the definition of the parabolic Airy process.

The proof relies on dedicated computations of the contour integrals appearing in the $\mathcal{D}$ function. While most of the computations are directly related to the formulas of the parabolic Airy process, we also use the following lemma in the proof, which is independent of the process and might be of its own interest. Note that when $m=1$, Lemma \ref{lm:Andreief_ext} is the well known Andreief's identity \cite{Andreief86}.

\begin{lm}[Generalized Andreief's Identity]
\label{lm:Andreief_ext}
Let $I_1,\ldots,I_m$ be a partition of $\{1,\ldots,n\}$. Denote by $\varphi$ the indicator function on $\cup_{k=1}^mI_k\times I_k$, i.e.,
\begin{equation}
\varphi(i,j) = \begin{dcases}
    1, & i,j\in I_k \text{ for some }k,\\
    0, & \text{elsewhere}.
\end{dcases}
\end{equation}

Let $X\subset\complexC$ be a measurable set and $\mu$ be a measure on $\Gamma$. Suppose $A_i(x)$ and $B_i(x)$, $1\le i\le n$, are two sequence of functions on $X$ such that $A_i(x)B_j(x)$, $1\le i,j \le n$, are all integrable functions with respect to $\rd \mu$. Then we have
\begin{equation}
\label{eq:Andeief_ext1}
\det \left[ \int_X A_i(x)B_j(x) \rd \mu(x)\right]_{i,j=1}^n
 = \frac{1}{\prod_{k=1}^m |I_k|!}  \int_{X^{n}}\det\left[A_i(x_j)\right]_{i,j=1}^n \det\left[B_i(x_j)\varphi(i,j)\right]_{i,j=1}^n \prod_{i=1}^n \rd \mu(x_i).
\end{equation}
Equivalently, we have
\begin{equation}
\label{eq:Andeief_ext2}
\det \left[ \int_X A_i(x)B_j(x) \rd \mu(x)\right]_{i,j=1}^n
 = \frac{1}{\prod_{k=1}^m |I_k|!}  \int_{X^{n}}\det\left[A_i(x_j)\right]_{i,j=1}^n \prod_{k=1}^m \det\left[B_i(x_j)\right]_{i,j\in I_k} \prod_{i=1}^n \rd \mu(x_i).
\end{equation}
\end{lm}
\begin{proof}[Proof of Lemma \ref{lm:Andreief_ext}]
The idea of the proof is similar to that of the Andreief's identity. We denote $S^*$ the set of permutations of $\{1,\ldots,n\}$ that map $I_k$ to itself for all $1\le k\le m$, i.e., $\sigma\in S^*$ if and only if $\sigma$ is a bijection of $\{1,\ldots,n\}$ and $\varphi(i,\sigma_i)=1$ for all $i$. It is direct to count that
\begin{equation}
\label{eq:card_S*}
|S^*| = \prod_{k=1}^m |I_k|!.
\end{equation}

Now we write
\begin{equation}
\begin{split}
     \det \left[ \int_X A_i(x)B_j(x) \rd \mu(x)\right]_{i,j=1}^n
    &= \int_{X^n} \det\left[ A_i(x_j) B_j(x_j)\right]_{i,j=1}^n \prod_{i=1}^n \rd \mu(x_i)\\
    &= \int_{X^n} \det\left[ A_i (x_j)\right]_{i,j=1}^n \prod_{i=1}^n B_i(x_i)\prod_{i=1}^n \rd \mu(x_i)\\
    &=\frac{1}{|S^*|} \sum_{\sigma\in S^*} \int_{X^n}\det \left[ A_i (x_{\sigma_j})\right]_{i,j=1}^n \prod_{i=1}^n B_i(x_{\sigma_i})\prod_{i=1}^n \rd \mu(x_i)\\
    &=\frac{1}{|S^*|} \sum_{\sigma\in S^*} \int_{X^n} \sgn(\sigma) \det\left[ A_i (x_j)\right]_{i,j=1}^n \prod_{i=1}^n B_i(x_{\sigma_i})\prod_{i=1}^n \rd \mu(x_i)\\
    &=\frac{1}{|S^*|}  \int_{X^n}  \det\left[ A_i (x_j)\right]_{i,j=1}^n \sum_{\sigma\in S^*} \sgn(\sigma)\prod_{i=1}^n B_i(x_{\sigma_i})\prod_{i=1}^n \rd \mu(x_i),
\end{split}
\end{equation}
where $\sgn(\sigma)$ denotes the sign of the permutation $\sigma$. Note that the last summation gives the product of the determinants $\det[B_i(x_j)]_{i,j\in I_k}$. Together with \eqref{eq:card_S*} we obtain \eqref{eq:Andeief_ext2}. The other identity \eqref{eq:Andeief_ext1} also follows by noting that the determinant $\det [B_i(x_j)\varphi(i,j)]_{i,j=1}^n$ equals to the product of the determinants of the blocks along the diagonal line, since the matrix $[B_i(x_j)\varphi(i,j)]_{i,j=1}^n$ is a block matrix with off-diagonal blocks all equal to zero.
\end{proof}

The structure of this paper is as follows. In section \ref{sec:Simplification}, we will show that there is a way to simplify the formula \eqref{eq:Airy_Liu}. More explicitly, the $z$-integrals can be evaluated, where the resulting formula can be further simplified as a Fredholm determinant in Proposition \ref{prop:alt_formula}. In section \ref{sec:Equivalence}, we show that the new Fredholm determinant formula matches \eqref{eq:def_AiryProcess}. We remark that we actually obtained several different formulas for the parabolic Airy process through the computations in this paper, see the equation \eqref{eq:sum_hatD}, Lemmas \ref{lm:z_integral}, \ref{lm:hatD_version2}, Proposition \ref{prop:alt_formula}, and equations \eqref{eq:bfL_expansion}, \eqref{eq:final_bfL}. We put forward Proposition \ref{prop:alt_formula} as a representative since it is the simplest formula that only involves the contour integrals.

\section{Contour integral formula of the parabolic Airy process}
\label{sec:Simplification}

In this section, we will first state the main result, a contour integral formula of the multipoint distribution of the parabolic Airy process in section \ref{sec:alt_formula}. See Proposition \ref{prop:alt_formula}. This formula is derived from the multipoint distribution formula of the KPZ fixed point by \cite{Liu22}. We will introduce the formula of \cite{Liu22} in section \ref{sec:Liu_formula}, and prove Proposition \ref{prop:alt_formula} in section \ref{sec:proof_simplification}.

\subsection{Contour integral formula for the parabolic Airy process}
\label{sec:alt_formula}

In this subsection, we introduce a multipoint distribution formula of the parabolic Airy process, which is given by a Fredholm determinant expansion in terms of contour integrals.

Assume that $m\ge 1$ is a fixed integer, and $\alpha_1,\ldots,\alpha_m,\beta_1,\ldots,\beta_m$ are fixed real number. We also assume that $\alpha_1<\cdots<\alpha_m$.

Define the functions
\begin{equation}
\label{eq:def_f}
f_i(w) = e^{-\frac13 w^3 +\alpha_i w^2 + \beta_i w}, \quad 1\le i\le m,
\end{equation}
and
\begin{equation}
\label{eq:def_F}
\begin{split}
F_i(w) &= \begin{dcases}
        f_1(w), & i=1,\\
          f_i(w)/f_{i-1}(w), &2\le i\le m
          \end{dcases}
          \\
        &
        =\begin{dcases}
        e^{-\frac13 w^3 +\alpha_1w^2 +\beta_1 w}, & i=1,\\
        e^{(\alpha_i -\alpha_{i-1})w^2 +(\beta_i -\beta_{i-1})w}, & 2\le i\le m.
        \end{dcases}
        \end{split}
\end{equation}

Denote $\complexC_\rL:=\{z\in\complexC: \Re(z)<0\}$ the left half of the complex plane. Define a function $\bfK: (\{1,\ldots,m\}\times \complexC_\rL)^2\to\complexC$ as follows,
\begin{equation}
\label{eq:def_bfK}
    \bfK(i,z; j,u) = \begin{dcases}
                       \int  \frac{\rd v}{2\pi\rmi} \frac{f_1(z)}{f_1(v)} \frac{1}{(z-v)(v-u)} , & i=1,\\
                       \int  \frac{\rd v}{2\pi\rmi} \prod_{\ell=2}^{i} \int  \frac{\rd u_\ell}{2\pi\rmi} \frac{f_1(z)}{f_i(v)}\frac{ \prod_{\ell=2}^{i}F_\ell(u_\ell)}{(z-u_2)\cdot \prod_{\ell=2}^{i-1}(u_\ell-u_{\ell+1})\cdot (u_i -v)(v-u)}, & 2\le i\le m,
                     \end{dcases}
\end{equation}
where the integration contour of $v$ is any contour on the right half plane that goes from $\infty e^{-\pi\rmi/5}$ to $\infty e^{\pi\rmi/5}$ \footnote{One could choose the angle of the $v$-contour to be $\pm\pi/3$, which is a standard choice for functions with a cubic exponent. We choose $\pm\pi/5$ here to be consistent with the rest of the paper, where we need to ensure not only $1/f_i(v)$, but also $1/F_i(v)$ which has a square exponent, to decay along the $v$-contour.}, and the integration contour of $u_\ell$ is a contour between $-\infty$ and $z$ on the left half plane that goes from $\infty e^{-2\pi\rmi/3}$ to $\infty e^{2\pi\rmi/3}$. Moreover, the integration contours of $u_2,\ldots,u_i$ are disjoint and ordered from right to left. See Figure \ref{fig:contours_bfK} for an illustration of the contours. Note that our choices of the integration contours ensures the integrals are absolutely convergent and hence well-defined.

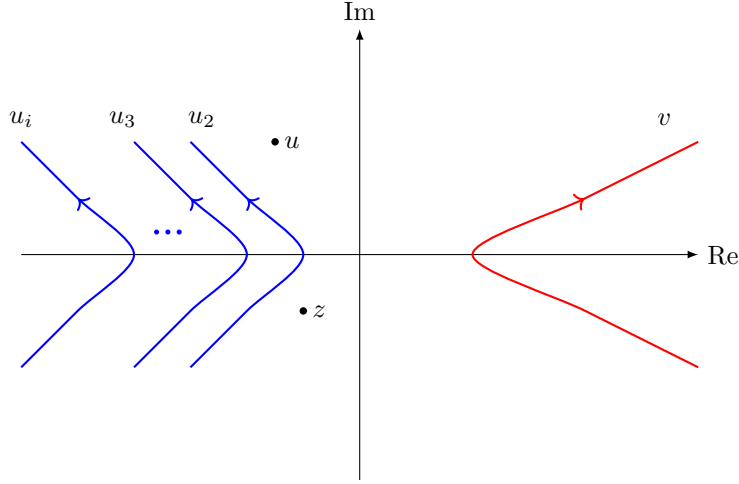
\begin{figure}
    \centering
    \begin{tikzpicture}[scale=1.5, decoration={
  markings,
  mark=at position 0.75 with {\arrow{>}}}
  ] 
  \draw[-latex] (-3,0) -- (3,0) node[right] {$\Re$};
  \draw[-latex] (0,-2) -- (0,2) node[above] {$\Im$};

  \draw[postaction={decorate},blue,thick] plot[smooth] coordinates {(-3,-1) (-2.5,-0.5) (-2,0) (-2.5,0.5) (-3,1)};
  \draw[postaction={decorate},blue,thick] plot[smooth] coordinates {(-2,-1) (-1.5,-0.5) (-1,0) (-1.5,0.5) (-2,1)};
  \draw[postaction={decorate},blue,thick] plot[smooth] coordinates {(-1.5,-1) (-1,-0.5) (-0.5,0) (-1,0.5) (-1.5,1)};

  \draw[postaction={decorate},red,thick] plot[smooth] coordinates {(3,-1) (2,-0.5) (1,0) (2,0.5) (3,1)};

  \filldraw (-0.75,1) circle (0.8pt) node[right] {$u$};
  \filldraw (-0.5,-0.5) circle (0.8pt) node[right] {$z$};
  \filldraw[blue] (-1.8,0.2) circle (0.5pt);
 \filldraw[blue] (-1.7,0.2) circle (0.5pt);
  \filldraw[blue] (-1.6,0.2) circle (0.5pt);
  \node at (-3,1.2) {$u_i$};
  \node at (-2.1,1.2) {$u_3$};
  \node at (-1.4,1.2) {$u_2$};
  \node at (2.7,1.2) {$v$};
\end{tikzpicture}
    \caption{The integration contours in the definition of $\bfK$.}
    \label{fig:contours_bfK}
\end{figure}

$\bfK$ can also be viewed as an operator on $L^2(\{1,\ldots,m\}\times \Gamma_\rL)$ if $\Gamma_\rL$ is a contour on the left half plane such that $f_1$ decays sufficiently fast along the contour. More explicitly, we choose
\begin{equation}
\Gamma_\rL =\left\{-1+ r e^{\pm 2\pi\rmi/3}: r\ge 0\right\},
\end{equation}
with the orientation  from $\infty e^{-2\pi\rmi/3}$ to $\infty e^{2\pi\rmi/3}$. With this choice, we define the Fredholm determinant $\det(\rI +\bfK)$ by its series expansion
\begin{equation}
\label{eq:Fredholm_series}
    \det (\rI +\bfK) = \sum_{k=0}^\infty \frac{1}{k!} \sum_{1\le \ell_1,\ldots,\ell_k\le m} \int_{\Gamma_\rL}\cdots\int_{\Gamma_\rL} \det\begin{bmatrix}
        \bfK(\ell_i,u_i;\ell_j,u_j) 
    \end{bmatrix}_{i,j=1}^k \frac{\rd u_1}{2\pi\rmi}\cdots\frac{\rd u_k}{2\pi\rmi},
\end{equation}
or equivalently by counting the number of $i$'s appearing in the first index in the above summation,
\begin{equation}
\label{eq:Fredholm_series2}
    \det (\rI +\bfK) =  \sum_{k_1,\ldots,k_m\ge 0} \frac{1}{k_1!\cdots k_m!}  \prod_{i=1}^m \prod_{\hat\ell_i=1}^{k_m}\int_{\Gamma_\rL} \frac{\rd u_{\hat \ell_i}^{(i)}}{2\pi\rmi} \det\begin{bmatrix}
        \bfK( i,u_{\hat \ell_i}^{(i)};j,u_{\hat\ell_j}^{(j)}) 
    \end{bmatrix}_{(i,\hat\ell_i),(j,\hat\ell_j)} ,
\end{equation}
where the row and column indices of the determinant above are chosen from the set $\{(i,\hat\ell_i): 1\le i\le m, 1\le \hat\ell_i\le k_i\}$. In the two formulas above, and other similar formulas in the rest of this paper, we view the empty product, integral, or determinant as $1$. Thus, the first term in both expansions is $1$. Moreover, it is standard to use the Hadamard's inequality to verify that the above multiple integral and the summation are absolutely convergent due to the fact that $\bfK$ decays super-exponentially fast along the contour $\Gamma_\rL$.

\begin{prop}
\label{prop:alt_formula}
Recall $\alpha_1<\cdots<\alpha_m$. We have the following formula for the $m$-point distribution of the parabolic Airy process $\Airy(\alpha)$
\begin{equation}
\prob\left(\bigcap_{i=1}^m \left\{ \Airy(\alpha_i) \le \beta_i \right\} \right) =\det(\rI + \bfK),
\end{equation}
where the Fredholm determinant $\det(\rI+\bfK)$ is defined in \eqref{eq:Fredholm_series}.
\end{prop}

When $m=1$, we set $\alpha_1=\alpha$ and $\beta_1=\beta$. The operator $\bfK$ is defined on $L^2(\Gamma_\rL, \rd z/2\pi\rmi)$ with kernel
\begin{equation}
\bfK(u;u') =\int_{\Gamma_\rR} \frac{e^{-\frac13 u^3 + \alpha u^2 +\beta u}}{e^{-\frac13v^3 +\alpha v^2 +\beta v}(u-v)(v-u')} \frac{\rd v}{2\pi\rmi} = \int_0^\infty L_1(u;\lambda)L_2(\lambda;u')\rd \lambda,
\end{equation}
where $\Gamma_\rR$ is any contour on the right half plane that goes from $\infty e^{-\pi\rmi/5}$ to $\infty e^{\pi\rmi/5}$,
\begin{equation}
\begin{split}
L_1(u;\lambda) & = \int_{\Gamma_\rR} \frac{e^{-\frac13u^3 +\alpha u^2 +\beta u}}{e^{-\frac13 v^3 +\alpha v^2 +(\beta+\lambda)v}} \frac{1}{u-v} \frac{\rd v}{2\pi\rmi}\\
&= -\int_0^\infty \int_{\Gamma_\rR} \frac{e^{-\frac13u^3 +\alpha u^2 +(\beta+\lambda') u}}{e^{-\frac13 v^3 +\alpha v^2 +(\beta+\lambda+\lambda')v}}  \frac{\rd v}{2\pi\rmi} \rd \lambda',
\end{split}
\end{equation}
and
\begin{equation}
L_2(\lambda;u') = e^{\lambda u'}. 
\end{equation}

Note that 
\begin{equation}
\begin{split}
(L_2L_1)(\lambda,\lambda'') &= -\int_0^\infty \int_{\Gamma_{\rL}} e^{-\frac13u^3 +\alpha u^2 +(\beta+\lambda+\lambda') u} \frac{\rd u}{2\pi\rmi} \int_{\Gamma_\rR}  e^{\frac13 v^3 -\alpha v^2 -(\beta+\lambda''+\lambda')v}  \frac{\rd v}{2\pi\rmi}   \rd \lambda'\\
&= -e^{\alpha(\lambda-\lambda'')}\int_0^\infty \Ai(\beta+\alpha^2+\lambda+\lambda')\Ai(\beta+\alpha^2+\lambda''+\lambda')\rd\lambda',
\end{split}
\end{equation}
which is the conjugated Airy kernel (up to the sign and parameter shift). Therefore, we obtain
\begin{equation}
\prob\left(\Airy(\alpha)\le\beta\right) = \det(\rI+\bfK)  =\det(\rI+L_2L_1) = \FGUE(\beta+\alpha^2),
\end{equation}
which recovers \eqref{eq:GUE_marginal}.

\subsection{Multipoint distribution formula of the KPZ fixed point}
\label{sec:Liu_formula}

In this subsection, we will introduce the multipoint distribution formula of the KPZ fixed point as discussed in Theorem \ref{thm:Liu}.

Assume that $(\alpha_i,\tau_i)$, $1\le i\le m$, are $m$ points in $\realR\times\realR_+$ satisfying $(\alpha_1,\tau_1)\prec \cdots \prec (\alpha_m,\tau_m)$. Then, the joint distribution of $\hKPZ (\alpha_i,\tau_i)$, $1\le i\le m$, is given by
\begin{equation}
\label{eq:Liu_formula}
     \prob\left(  \bigcap_{i =1}^ m \left\{ \hKPZ (\alpha_i,\tau_i) \le  \beta_i \right\}\right)
     = \oint_0 \cdots \oint_0 \mathcal{D}_{(\alpha_1,\tau_1),\ldots,(\alpha_m,\tau_m)}(\bsz; \beta_1,\ldots, \beta_m) \prod_{i=1}^{m-1} \frac{\rd z_i}{2\pi\rmi z_i(1-z_i)},
\end{equation}
where $\bsz =(z_1,\ldots,z_{m-1})$, and $\mathcal{D}_{(\alpha_1,\tau_1),\ldots,(\alpha_m,\tau_m)}(\bsz; \beta_1,\ldots, \beta_m)$ is a function explicitly defined in \cite{Liu22}. Below in this subsection we suppress the parameters and write $\mathcal{D}=\mathcal{D}_{(\alpha_1,\tau_1),\ldots,(\alpha_m,\tau_m)}(\bsz; \beta_1,\ldots, \beta_m)$ for notation simplification when there is no confusion.

There are two equivalent formulas of $\mathcal{D}$ in \cite{Liu22}, one in the form of a Fredholm determinant of an operator, and the other in the form of a series expansion (of the Fredholm determinant). We will introduce the expansion formula, since we will not use the operator form of $\mathcal{D}$ in this paper.

\bigskip

We first define the functions
\begin{equation}
f_i(w) = e^{-\frac13\tau_i w^3 +\alpha_i w^2 +\beta_i w},\quad 1\le i\le m,
\end{equation}
and
\begin{equation}
\begin{split}
F_i(w) & = \begin{dcases}
                f_1(w), & i=1,\\
                f_i(w)/f_{i-1}(w), &2\le i\le m,
           \end{dcases}\\
       & = \begin{dcases}
                e^{-\frac13 \tau_1 w^3 + \alpha_1 w^2 +\beta_1 w}, & i=1,\\
                e^{-\frac13 (\tau_i -\tau_{i-1})w^3 + (\alpha_i -\alpha_{i-1})w^2 +(\beta_i-\beta_{i-1})w}, & 2\le i\le m.
           \end{dcases}
\end{split}
\end{equation}
These notations are consistent with \eqref{eq:def_f} and \eqref{eq:def_F} when $\tau_1=\cdots=\tau_m=1$.

We also introduce the notation of Cauchy determinant
\begin{equation}
\rC(W;\tilde W) := \det \left[ \frac{1}{w_i-\tilde w_{j}}\right]_{i,j=1}^n = (-1)^{n(n-1)/2}\frac{\prod_{1\le i<j\le n} (w_j-w_i)(\tilde w_j-\tilde w_i)}{\prod_{i=1}^n\prod_{j=1}^n (w_j-\tilde w_i)}
\end{equation}
for any $n\ge 1$, and any vectors $W=(w_1,\ldots,w_n), \tilde W=(\tilde w_1,\ldots,\tilde w_n)\in\complexC^n$. Note that in the definition, the dimensions of $W$ and $\tilde W$ have to be the same. 

Another related notation we need is the following operation $\sqcup$ of two vectors. If $W=(w_1,\ldots,w_n)\in\complexC^n$ and $W'=(w'_1,\ldots,w'_{n'})\in\complexC^{n'}$ are two vectors, define their conjunction to be
\begin{equation}
W\sqcup W' = (w_1,\ldots,w_n,w'_1,\ldots,w'_{n'})\in\complexC^{n+n'}.
\end{equation}

Finally, we introduce $4m-2$ contours as follows. Let $\Gamma_{m,\rL}^\inn,\ldots,\Gamma_{2,\rL}^\inn,\Gamma_{1,\rL},\Gamma_{2,\rL}^\out,\ldots,\Gamma_{m,\rL}^\out$, ordered from left to right, be $2m-1$ contours on the left half of the complex plane. Each of them goes from $\infty e^{-2\pi\rmi/3}$ to $\infty e^{2\pi\rmi/3}$. Moreover, let $\Gamma_{m,\rR}^\inn,\ldots,\Gamma_{2,\rR}^\inn,\Gamma_{1,\rR},\Gamma_{2,\rR}^\out,\ldots,\Gamma_{m,\rR}^\out$, ordered from right to left, be $2m-1$ contours on the right half of the complex plane. Each of them goes from $\infty e^{-\pi\rmi/5}$ to $\infty e^{\pi\rmi/5}$. Note that the angles of these right contours are $\pm\pi/5$ instead of $\pm\pi/3$ \footnote{When the time parameters are strictly ordered $\tau_1<\cdots<\tau_m$, we could choose the angles of the right contours to be $\pm\pi/3$. When the time parameters are not necessarily strictly different, we need to bend the right contours to ensure the convergence of the integrals in the formula. See \cite[discussions after Definition 2.25]{Liu22}, or more recently \cite[Proposition 3.1]{Liu-Zhang25}.}.  See Figure \ref{fig:Gamma_contours} for an illustration of the contours.

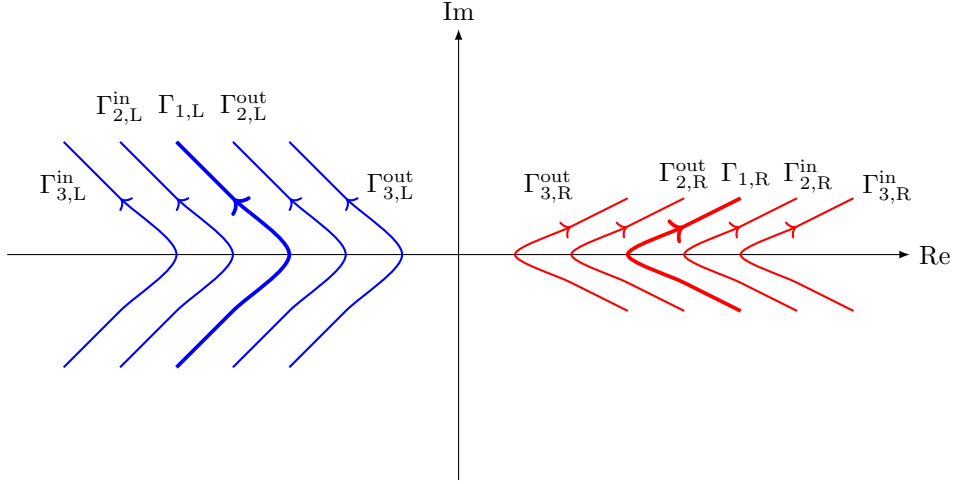
\begin{figure}
    \centering
    \begin{tikzpicture}[scale=1.5, decoration={
  markings,
  mark=at position 0.75 with {\arrow{>}}}
  ] 
  \draw[-latex] (-4,0) -- (4,0) node[right] {$\Re$};
  \draw[-latex] (0,-2) -- (0,2) node[above] {$\Im$};

    \draw[postaction={decorate},blue,thick] plot[smooth] coordinates {(-3.5,-1) (-3,-0.5) (-2.5,0) (-3,0.5) (-3.5,1)};
  \draw[postaction={decorate},blue,thick] plot[smooth] coordinates {(-3,-1) (-2.5,-0.5) (-2,0) (-2.5,0.5) (-3,1)};
  \draw[postaction={decorate},blue,line width = 0.5mm] plot[smooth] coordinates {(-2.5,-1) (-2,-0.5) (-1.5,0) (-2,0.5) (-2.5,1)};
  \draw[postaction={decorate},blue,thick] plot[smooth] coordinates {(-2,-1) (-1.5,-0.5) (-1,0) (-1.5,0.5) (-2,1)};
  \draw[postaction={decorate},blue,thick] plot[smooth] coordinates {(-1.5,-1) (-1,-0.5) (-0.5,0) (-1,0.5) (-1.5,1)};

  \draw[postaction={decorate},red,thick] plot[smooth] coordinates {(1.5,-0.5) (1,-0.25) (0.5,0) (1,0.25) (1.5,0.5)};
  \draw[postaction={decorate},red,thick] plot[smooth] coordinates {(2,-0.5) (1.5,-0.25) (1,0) (1.5,0.25) (2,0.5)};
  \draw[postaction={decorate},red,line width =0.5mm] plot[smooth] coordinates {(2.5,-0.5) (2,-0.25) (1.5,0) (2,0.25) (2.5,0.5)};
   \draw[postaction={decorate},red,thick] plot[smooth] coordinates {(3,-0.5) (2.5,-0.25) (2,0) (2.5,0.25) (3,0.5)};
\draw[postaction={decorate},red,thick] plot[smooth] coordinates {(3.5,-0.5) (3,-0.25) (2.5,0) (3,0.25) (3.5,0.5)};

  \node at (-3.5,0.6) {$\Gamma_{3,\rL}^{\inn}$};
  \node at (-3,1.3) {$\Gamma_{2,\rL}^{\inn}$};
  \node at (-2.45,1.3) {$\Gamma_{1,\rL}$};
  \node at (-1.9,1.3) {$\Gamma_{2,\rL}^\out$};
  \node at (-0.6,0.6) {$\Gamma_{3,\rL}^\out$};

  \node at (3.8,0.6) {$\Gamma_{3,\rR}^{\inn}$};
  \node at (2.55,0.7) {$\Gamma_{1,\rR}$};
  \node at (2,0.7) {$\Gamma_{2,\rR}^\out$};
  \node at (3.1,0.7) {$\Gamma_{2,\rR}^\inn$};
  \node at (0.8,0.6) {$\Gamma_{3,\rR}^\out$};
\end{tikzpicture}
    \caption{Illustration of the $\Gamma$-contours when $m=3$. The two contours $\Gamma_{1,\rL}$ and $\Gamma_{1,\rR}$ are thickened.}
    \label{fig:Gamma_contours}
\end{figure}

Now we are ready to define the integrand $\mathcal{D}$.
\begin{defn}
\label{def:calD}
The function $\mathcal{D}=\mathcal{D}_{(\alpha_1,\tau_1),\ldots,(\alpha_m,\tau_m)}(\bsz; \beta_1,\ldots, \beta_m)$ is defined by the following series expansion
\begin{equation}
\label{eq:D_sum}
\mathcal{D}= \sum_{n_1,\ldots,n_m\ge 0} \frac{1}{(n_1!\cdots n_m!)^2} \mathcal{D}_{\bsn},
\end{equation}
where $\bsn=(n_1,\cdots,n_m)$, and 
\begin{equation}
\label{eq:D_individual}
\begin{split}
\mathcal{D}_\bsn  &= \mathcal{D}_{\bsn;(\alpha_1,\tau_1),\ldots,(\alpha_m,\tau_m)}(\bsz; \beta_1,\ldots, \beta_m)\\
&:= \prod_{i=1}^{m-1} (1-z_i)^{n_i} (1-z_i^{-1})^{n_{i+1}}\\
&\quad \cdot \prod_{i=2}^m \prod_{\ell_i=1}^{n_i} 
     \left(\frac{1}{1-z_{i-1}} \int_{\Gamma_{i,\rL}^\inn} \frac{\rd u_{\ell_i}^{(i)}}{2\pi\rmi} 
     -\frac{z_{i-1}}{1-z_{i-1}} \int_{\Gamma_{i,\rL}^\out} \frac{\rd u_{\ell_i}^{(i)}}{2\pi\rmi} \right)
     \prod_{\ell_1=1}^{n_1} \int_{\Gamma_{1,\rL}} \frac{\rd u_{\ell_1}^{(1)}}{2\pi\rmi}\\
     &\quad \cdot \prod_{i=2}^m \prod_{\ell_i=1}^{n_i} 
     \left(\frac{1}{1-z_{i-1}} \int_{\Gamma_{i,\rR}^\inn} \frac{\rd v_{\ell_i}^{(i)}}{2\pi\rmi} 
     -\frac{z_{i-1}}{1-z_{i-1}} \int_{\Gamma_{i,\rR}^\out} \frac{\rd v_{\ell_i}^{(i)}}{2\pi\rmi} \right)
     \prod_{\ell_1=1}^{n_1} \int_{\Gamma_{1,\rR}} \frac{\rd v_{\ell_1}^{(1)}}{2\pi\rmi}\\
     &\qquad \rC(V^{(1)};U^{(1)}) \cdot \prod_{i=1}^m \rC(U^{(i)}\sqcup V^{(i+1)}; V^{(i)}\sqcup U^{(i+1)}) \cdot  \prod_{i=1}^m \prod_{\ell_i=1}^{n_i} \frac{F_i (u_{\ell_i}^{(i)})}{F_i (v_{\ell_i}^{(i)})},
\end{split}
\end{equation}
where the vectors $U^{(i)} = (u_1^{(i)},\cdots, u_{n_i}^{(i)})$, $V^{(i)} = (v_1^{(i)},\cdots, v_{n_i}^{(i)})$ for $1\le i \le m$, and $U^{(m+1)}$, $V^{(m+1)}$ are both empty vectors.
\end{defn}

It is direct to check that $F_i(u)$ (and $1/F_i(v)$, respectively) decays super-exponentially fast as $u\in\Gamma_{i,\rL}^{\inn}\cup \Gamma_{i,\rL}^\out$ ($v\in \Gamma_{i,\rR}^{\inn}\cup \Gamma_{i,\rR}^\out$, respectively) goes to infinity when $i\ge 2$, or $u\in\Gamma_{1,\rL}$ ($v\in \Gamma_{1,\rR}$, respectively) goes to infinity when $i=1$. One can show that the integrals in $\mathcal{D}_{\bsn}$ are absolutely convergent, and the summation in \eqref{eq:D_sum} is absolutely convergent uniformly for $(z_1,\ldots,z_{m-1})$ as long as all the norms $|z_i|$, $1\le i\le m$, stay within a compact set of $(0,1)$.  Similar expansions have been discussed in \cite{Liu22,Liu22c,Liu-Wang24,Liu-Zhang25}. We refer the readers to the discussions after Definition 2.2 in \cite{Liu-Zhang25} for more details of a similar convergence.

\subsection{Proof of Proposition \ref{prop:alt_formula} using Theorem \ref{thm:Liu}}
\label{sec:proof_simplification}

In this subsection, we consider a special case of the formula \eqref{eq:Liu_formula} when $\tau_1=\cdots=\tau_m=1$ and $\alpha_1<\cdots<\alpha_m$. After simplifying and rewriting the formula, we will show Proposition \ref{prop:alt_formula}.

Recall the discussions at the end of the previous subsection. We can change the order of the integration and the summation. Consider the $z$-integrals inside the summation \eqref{eq:D_sum} and denote
\begin{equation}
\hat{\mathcal{D}}_{\bsn}  = \hat{\mathcal{D}}_{\bsn;\alpha_1,\ldots,\alpha_m}(\beta_1,\ldots,\beta_m):= \oint_0\cdots \oint_0 \mathcal{D}_{\bsn;(\alpha_1,1),\ldots,(\alpha_m,1)}(\bsz;\beta_1,\ldots,\beta_m) \prod_{i=1}^{m-1}\frac{\rd z_i}{2\pi\rmi z_i(1-z_i)}.
\end{equation}
Then we have
\begin{equation}
\label{eq:sum_hatD}
\prob\left(\bigcap_{i=1}^m \left\{\Airy(\alpha_i)\le \beta_i \right\}\right) = \sum_{n_1,\ldots,n_m\ge 0} \frac{1}{(n_1!\cdots n_m!)^2} \hat{\mathcal{D}}_{\bsn}.
\end{equation}

Below in this subsection, we will first evaluate the $z$-integrals by simplifying the $v$-integrals in $\hat{\mathcal{D}}_{\bsn}$. This gives Lemma \ref{lm:z_integral}. We further re-organize the $u$-integrals in the formula and give a more compact formula for $\hat{\mathcal{D}}_{\bsn}$ in Lemma \ref{lm:hatD_version2}. Finally, we will show the summation in \eqref{eq:sum_hatD} is the same as $\det(\rI+ \bfK)$ in Proposition \ref{prop:alt_formula}.

\bigskip

For notation convenience, we set 
\begin{equation}
    \Gamma_{1,\rL}^\inn =\Gamma_{1,\rL},\quad \Gamma_{1,\rR}^\out =\Gamma_{1,\rR}
\end{equation}
throughout this subsection.

\begin{lm}
\label{lm:z_integral}
    Recall that $\alpha_1<\cdots<\alpha_m$. 
    \begin{enumerate}[(i)]
    \item We have
    \begin{equation}
    \label{eq:hat_D_expression}
    \begin{split}       
       \hat{\mathcal{D}}_{\bsn} &= \prod_{i=1}^m \prod_{\ell_i=1}^{n_i}  \left(\int_{\Gamma_{i,\rL}^\inn} \frac{\rd u_{\ell_i}^{(i)}}{2\pi\rmi}  \int_{\Gamma_{i,\rR}^\out} \frac{\rd v_{\ell_i}^{(i)}}{2\pi\rmi}  \right)\\
        &\quad
        \rC(V^{(1)};U^{(1)}) \cdot \prod_{i=1}^m \rC(U^{(i)}\sqcup V^{(i+1)}; V^{(i)}\sqcup U^{(i+1)}) \cdot  \prod_{i=1}^m \prod_{\ell_i=1}^{n_i} \frac{F_i (u_{\ell_i}^{(i)})}{F_i (v_{\ell_i}^{(i)})},
    \end{split}
    \end{equation}
    where all the notations are the same as in the previous subsection. 

    \item  If $n_1\ge \cdots\ge n_m$, we have
    \begin{equation}
    \label{eq:hat_D_simplification}
    \begin{split}
      \hat{\mathcal{D}}_{\bsn} 
      & = \prod_{i=1}^{m-1}\frac{n_i!}{k_i!}   \prod_{i=1}^m \prod_{\ell_i=1}^{n_i}   \int_{\Gamma_{i,\rL}^\inn} \frac{\rd u_{\ell_i}^{(i)}}{2\pi\rmi}\cdot \prod_{i=1}^m \prod_{\hat\ell_i=1}^{k_i} \int_{\Gamma_{1,\rR}} \frac{\rd \hat v_{\ell_i}^{(i)}}{2\pi\rmi} \\
      &\quad \rC(\hat V^{(m)}\sqcup \cdots \sqcup \hat V^{(1)};U^{(1)}) \cdot \prod_{i=1}^m \rC(U^{(i)}; U^{(i+1)}\sqcup \hat V^{(i)}) \cdot \prod_{i=1}^m \frac{\prod_{\ell_i=1}^{n_i}F_i(u_{\ell_i}^{(i)})}{\prod_{\hat\ell_i=1}^{k_i} f_i(\hat v_{\hat\ell_i}^{(i)})},
    \end{split}
    \end{equation}
    where $k_i = n_i -n_{i+1}$ for $1\le i\le m$, with the convention that $n_{m+1}=0$, and $\hat V^{(i)} = (\hat v_1^{(i)},\ldots,\hat v_{k_i}^{(i)})$ for $1\le i\le m$, and the functions $f_i$ are defined in \eqref{eq:def_f}. Moreover, for all other $\bsn$ we have $\hat{\mathcal{D}}_{\bsn}=0$.
    \end{enumerate}
\end{lm}

\begin{proof}[Proof of Lemma \ref{lm:z_integral}]

The proof of part (i) is basically the same as that of Lemma 6.1 in \cite{Liu-Zhang25}. We look at the expansion of the $v$-integrals in \eqref{eq:D_individual} by only taking one contour for each variable in one term. We give an example to illustrate such an expansion: $(c_1\int_{I_1}\rd x +c_2\int_{I_2} \rd x)(c'_1\int_{I'_1}\rd x' +c'_2\int_{I'_2} \rd x')$ can be expanded to the sum of the following $4$ terms $c_ic'_{i'}\int_{I_i}\rd x\int_{I'_{i'}}\rd x'$, $1\le i,i'\le 2$. After the expansion, we have a sum of terms each of which is a multiple contour integral with each $v$-variable running over one single contour. Consider a term where each $v$-contour is a single contour. If any of the $v_{\ell_i}^{(i)}$ contour is $\Gamma_{i,\rR}^\inn$, assume $i\ge 2$ is the largest index with this property. With this assumption, we can deform the $v_{\ell_i}^{(i)}$ contour from  $\Gamma_{i,\rR}^\inn$ to $R+\Gamma_{i,\rR}^\inn$ for any $R>0$ without encountering any pole since the only possible poles are $u_{\ell_{j}}^{(j)}$'s on the left half plane, $v_{\ell_{i-1}}^{(i-1)} \in \Gamma_{i-1,\rR}^{\inn}\cup\Gamma_{i-1,\rR}^{\out}$, and $v_{\ell_{i+1}}^{(i+1)}\in\Gamma_{i+1,\rR}^\out$ which are all on the left side of $\Gamma_{i,\rR}^\inn$, see Figure \ref{fig:Gamma_contours}. Now we let $R\to +\infty$ and the $v_{\ell_i}^{(i)}$ integral decays to zero since $1/F_i(v_{\ell_i}^{(i)})$ decays super-exponentially fast as $R$ grows for any $i\ge 2$, see the definition of $F_i$ functions in \eqref{eq:def_F}. Thus, the integral vanishes if any $\Gamma_{i,\rR}^\inn$ is included in the term from the expansion of \ \eqref{eq:D_individual}. Now we drop all these vanishing terms and have only one surviving term 
\begin{equation}
\mathcal{D}_\bsn = \prod_{i=1}^{m-1} (1-z_i)^{n_i}\prod_{i=2}^m \prod_{\ell_i=1}^{n_i} 
     \left(\frac{1}{1-z_{i-1}} \int_{\Gamma_{i,\rL}^\inn} \frac{\rd u_{\ell_i}^{(i)}}{2\pi\rmi} 
     -\frac{z_{i-1}}{1-z_{i-1}} \int_{\Gamma_{i,\rL}^\out} \frac{\rd u_{\ell_i}^{(i)}}{2\pi\rmi} \right)\int_{\Gamma_{i,\rR}^\out} \frac{\rd v_{i_\ell}^{(i)}}{2\pi\rmi}\cdots,
\end{equation}
where we suppressed the integrand, which is irrelevant to the $z$-variables. The function above is analytic in $z_i$ at the origin for each $i$. Hence, evaluating the $z$-integrals $\oint_0\cdots \oint_0 \mathcal{D}_{\bsn} \prod_{i=1}^{m-1}\frac{\rd z_i}{2\pi\rmi z_i(1-z_i)}$ is the same as inserting $z_i=0$ in the expression above. This proves \eqref{eq:hat_D_expression}.

Now we prove the second part of the lemma. It requires evaluations of some $v$-integrals recurrently, which we explain below.

We start with the formula \eqref{eq:hat_D_expression} and evaluate the integrals of $v_{\ell_m}^{(m)}$ along the contour $\Gamma_{m,\rR}^\out$. If we deform one $v_{\ell_m}^{(m)}$ contour to $\Gamma_{m,\rR}^\inn$, then the $v_{\ell_m}^{(m)}$-integral along the new contour $\Gamma_{m,\rR}^\inn$ vanishes, as we discussed in the proof of part (i). This implies that each integral of $v_{\ell_m}^{(m)}$ along the contour $\Gamma_{m,\rR}^\out$ only leaves the residue of the integrand when $v_{\ell_m}^{(m)}$ encounters one of the poles $v_{\ell_{m-1}}^{(m-1)}$, $1\le \ell_{m-1}\le n_{m-1}$. Moreover, due to the Cauchy determinant structure, the factor $\rC(U^{(m)};V^{(m)})$ vanished when we evaluate two $v_{\ell_m}^{(m-1)}$ variables at the same pole (e.g., when both $v_{1}^{(m)}$ and $v_2^{(m)}$ equals to $v_1^{(m-1)}$). This implies that the integrals of $v_1^{(m)},\ldots,v_{n_m}^{(m)}$ give zero (hence $\hat{\mathcal{D}}_\bsn=0$) unless $n_m\le n_{m-1}$ and $v_{1}^{(m)},\dots,v_{n_m}^{(m)}$ are evaluated at different poles among $v_1^{(m-1)},\ldots,v_{n_{m-1}}^{(m-1)}$. Also note that the integrand is symmetric on the poles $v_1^{(m-1)},\ldots,v_{n_{m-1}}^{(m-1)}$. Therefore, there are $n_{m-1}\cdot (n_{m-1}-1) \cdots (n_{m-1}-n_m+1) = {n_{m-1}!}/{k_{m-1}!}$ ways to evaluate the residues and each way gives the same result. Without loss of generality, we evaluate the pole of $v_{\ell_m}^{(m)}$ at $v_{\ell_m}^{(m-1)}$ for each $\ell_m=1,\ldots,n_{m}$. This evaluation gives
\begin{equation}
\label{eq:aux_01_lemma}
\begin{split}
    \hat{\mathcal{D}}_{\bsn} &= \frac{n_{m-1}!}{k_{m-1}!} \prod_{i=1}^{m-1} \prod_{\ell_i=1}^{n_i}  \left(\int_{\Gamma_{i,\rL}^\inn} \frac{\rd u_{\ell_i}^{(i)}}{2\pi\rmi}  \int_{\Gamma_{i,\rR}^\out} \frac{\rd v_{\ell_i}^{(i)}}{2\pi\rmi}  \right) \prod_{\ell_m=1}^{n_m}\int_{\Gamma_{m,\rL}^\inn} \frac{\rd u_{\ell_m}^{(m)}}{2\pi\rmi} \\
        &\quad
        \rC(V^{(1)};U^{(1)}) \cdot \prod_{i=1}^{m-2} \rC(U^{(i)}\sqcup V^{(i+1)}; V^{(i)}\sqcup U^{(i+1)}) \cdot \rC\left(U^{(m-1)}; U^{(m)}\sqcup V^{(m-1)}_{[n_m+1,n_{m-1}]}\right)\\
        &\quad\cdot   \rC\left(U^{(m)};V^{(m-1)}_{[1,n_{m}]}\right) \cdot  \frac{\prod_{i=1}^{m} \prod_{\ell_i=1}^{n_i} F_i (u_{\ell_i}^{(i)})}{ \prod_{i=1}^{m-1} \prod_{\ell_i=1}^{n_i} F_i (v_{\ell_i}^{(i)}) \cdot  \prod_{\ell_{m}=1}^{n_m} F_m(v_{\ell_{m}}^{(m-1)})},
\end{split}
\end{equation}
where we used the following notation
\begin{equation}
\label{eq:notation_partV}
V_{[a,b]}^{(i)}:= (v_a^{(i)},v_{a+1}^{(i)},\ldots,v_b^{(i)}), \quad 1\le a\le b\le n_i,\quad 1\le i\le m.
\end{equation}
We continue the same argument for the $v_{\ell_{m-1}}^{(m-1)}$ integrals. Similarly, we get $\hat{\mathcal{D}}_{\bsn}=0$ if $n_{m-1}>n_{m-2}$. Moreover, when $n_{m-1}\le n_{m-2}$, the residue evaluation gives
\begin{equation}
\begin{split}
    \hat{\mathcal{D}}_{\bsn} &= \frac{n_{m-1}!n_{m-2}!}{k_{m-1}!k_{m-2}!} \prod_{i=1}^{m-2} \prod_{\ell_i=1}^{n_i}  \left(\int_{\Gamma_{i,\rL}^\inn} \frac{\rd u_{\ell_i}^{(i)}}{2\pi\rmi}  \int_{\Gamma_{i,\rR}^\out} \frac{\rd v_{\ell_i}^{(i)}}{2\pi\rmi}  \right) \prod_{j=m-1}^m\prod_{\ell_j=1}^{n_j}\int_{\Gamma_{j,\rL}^\inn} \frac{\rd u_{\ell_j}^{(j)}}{2\pi\rmi} \\
        &\quad
        \rC(V^{(1)};U^{(1)}) \cdot \prod_{i=1}^{m-3} \rC(U^{(i)}\sqcup V^{(i+1)}; V^{(i)}\sqcup U^{(i+1)}) \cdot \prod_{j=m-2}^{m-1}\rC\left(U^{(j)}; U^{(j+1)}\sqcup V^{(m-2)}_{[n_{j+1}+1,n_{j}]}\right)\\
        &\quad   \cdot   \rC\left(U^{(m)};V^{(m-2)}_{[1,n_{m}]}\right) \cdot  \frac{\prod_{i=1}^{m} \prod_{\ell_i=1}^{n_i} F_i (u_{\ell_i}^{(i)})}{ \prod_{i=1}^{m-2} \prod_{\ell_i=1}^{n_i} F_i (v_{\ell_i}^{(i)}) \cdot  \prod_{j=m-1}^m\prod_{\ell_{j}=1}^{n_j} F_j(v_{\ell_j}^{(m-2)})}.
\end{split}
\end{equation}
Repeating this procedure until only $v_{\ell_1}^{(1)}$ integrals are left, we obtain $\hat{\mathcal{D}}_{\bsn}=0$ if any $n_i>n_{i-1}$ holds, $2\le i\le m$. Moreover, if $n_1\ge \cdots\ge n_m$, we end with
\begin{equation}
\label{eq:aux_02_lemma}
\begin{split}
    \hat{\mathcal{D}}_{\bsn} &= \prod_{i=1}^{m-1}\frac{n_{i}!}{k_{i}!} \prod_{j=2}^m\prod_{\ell_j=1}^{n_j}\int_{\Gamma_{j,\rL}^\inn} \frac{\rd u_{\ell_j}^{(j)}}{2\pi\rmi} 
        \prod_{\ell_1=1}^{n_1} \left(\int_{\Gamma_{1,\rL}} \frac{\rd u_{\ell_1}^{(1)}}{2\pi\rmi} \int_{\Gamma_{1,\rR}} \frac{\rd v_{\ell_1}^{(1)}}{2\pi\rmi} \right)\\
        &\quad
        \rC(V^{(1)};U^{(1)}) \cdot \prod_{j=1}^{m-1}\rC\left(U^{(j)}; U^{(j+1)}\sqcup V^{(1)}_{[n_{j+1}+1,n_{j}]}\right)    \cdot   \rC\left(U^{(m)};V^{(1)}_{[1,n_{m}]}\right) \cdot  \frac{\prod_{i=1}^{m} \prod_{\ell_i=1}^{n_i} F_i (u_{\ell_i}^{(i)})}{   \prod_{j=1}^m\prod_{\ell_{j}=1}^{n_j} F_j(v_{\ell_j}^{(1)})}.
\end{split}
\end{equation}
Note that, by recalling \eqref{eq:def_f} and \eqref{eq:def_F},
\begin{equation}
\prod_{j=1}^m\prod_{\ell_{j}=1}^{n_j} F_j(v_{\ell_j}^{(1)}) =\prod_{j=1}^{m} \prod_{\ell_j=n_{j+1}+1}^{n_j} f_j(v_{\ell_j}^{(1)}).
\end{equation}
By relabeling the $v^{(1)}$-variables in \eqref{eq:aux_02_lemma}, we obtain \eqref{eq:hat_D_simplification}.
\end{proof}

\bigskip

The next step is to further simplify \eqref{eq:hat_D_simplification} by re-organizing the $u$-integrals. We will assume that $n_1\ge \cdots\ge n_m$, since otherwise $\hat{\mathcal{D}}_\bsn=0$ as shown in Lemma \ref{lm:z_integral}.

\begin{lm}
\label{lm:hatD_version2}
Assume $k_i=n_i-n_{i+1}\ge 0$ for each $i=1,\ldots,m$, where we set $n_{m+1}=0$ for notation convenience. Then we have
\begin{equation}
\label{eq:hatD_expression2}
\begin{split}
    \hat{\mathcal{D}}_\bsn & = \left(\prod_{i=1}^{m-1}\frac{n_i!}{k_i!}\right)^2 \prod_{i=1}^m \prod_{\hat \ell_i=1}^{k_i} \int_{\Gamma_{1,\rL}} \frac{\rd \hat u_{\hat \ell_i}^{(i)}}{2\pi\rmi}
    \int_{\Gamma_{1,\rR}} \frac{\rd \hat v_{\hat \ell_i}^{(i)}}{2\pi\rmi}\rC(\hat V; \hat U) \cdot \prod_{i=1}^m \det\left[h_i(\hat u^{(i)}_{a},\hat v^{(i)}_b)\right]_{a,b=1}^{k_i} \cdot \prod_{i=1}^m \prod_{\hat \ell_i=1}^{k_i} \frac{1}{f_i(\hat v_{\hat \ell_i}^{(i)})}
\end{split}
\end{equation}
where 
\begin{equation}
\hat V= \hat V^{(m)}\sqcup \cdots \sqcup \hat V^{(1)}, \hat U = \hat U^{(m)}\sqcup \cdots \sqcup \hat U^{(1)}
\end{equation}
with 
\begin{equation}
\hat U^{(i)} = (\hat u_1^{(i)},\ldots, \hat u_{k_i}^{(i)}),\quad \hat V^{(i)} = (\hat v_1^{(i)},\ldots,\hat v_{k_i}^{(i)}),\quad 1\le i\le m,
\end{equation}
and the functions $h_i(u,v)$ are defined for all $(u,v)\in\Gamma_{1,\rL}\times\Gamma_{1,\rR}$ as follows
\begin{equation}
\label{eq:def_hi}
h_i(u,v)= 
\begin{dcases}
\frac{F_1(u)}{u-v},& i=1,\\
\prod_{j=2}^{i}\int_{\Gamma_{j,\rL}^\inn}  \frac{\rd u_j}{2\pi\rmi}  \frac{\prod_{j=2}^{i}F_{j}(u_j) \cdot F_1(u)}{(u-u_2)\cdot \prod_{j=2}^{i-1}(u_j-u_{j+1})\cdot  (u_i-v)},& 2\le i\le m.
\end{dcases}
\end{equation}
\end{lm}

We will use the following simple lemma repeatedly in the proof of Lemma \ref{lm:hatD_version2}. 
\begin{lm}
\label{lm:symmetry}
Let $\Sigma$ be a contour. $F$ is a function on $\Sigma^n$ which is antisymmetric, i.e., $F(w_1,\ldots,w_n) =\sgn(\sigma) F(w_{\sigma_1},\ldots,w_{\sigma_n})$ for any permutation $\sigma$ of $\{1,\ldots,n\}$. Assume that the integrals in this lemma are all well-defined.
\begin{enumerate}[(i)]
\item Let $p_i$, $1\le i\le n$, be functions. Then
\begin{equation}
\label{eq:identity_expansion_u}
\int_{\Sigma^n} F(w_1,\ldots,w_n) \det\left[ p_i(w_j)\right]_{i,j=1}^n \prod_{i=1}^n \rd w_i = n! \int_{\Sigma^n} F(w_1,\ldots,w_n) \prod_{i=1}^n p_i(w_i) \rd w_i.
\end{equation}
\item Assume $k\le n$. Let $w'_j$, $1\le j\le k$, be complex numbers on a contour $\Sigma'$, and $q$ be a function on $\Sigma\times\Sigma'$. Then 
\begin{equation}
\int_{\Sigma^n} F(w_1,\ldots,w_n) \prod_{j=1}^{k} q(w_j, w'_j)  \prod_{i=1}^n\rd w_i
\end{equation}
is antisymmetric in $w'_1,\ldots,w'_k$.
\end{enumerate}
\end{lm}
\begin{proof}[Proof of Lemma \ref{lm:symmetry}]
For (i), one just needs to expand the determinant and use the antisymmetry of $F$. Now consider the second part. For any permutation of $\sigma:\{1,\ldots,k\}\to \{1,\ldots,k\}$, we extend it to a permutation on $\{1,\ldots,n\}$ by defining $\sigma(i)=i$ if $i>k$. Then
\begin{equation}
\begin{split}
    \int_{\Sigma^n} F(w_1,\ldots,w_n) \prod_{j=1}^{k} q(w_j, w'_{\sigma_j})  \prod_{i=1}^n\rd w_i
    &=\int_{\Sigma^n} F(w_{\sigma_1},\ldots,w_{\sigma_n}) \prod_{j=1}^{k} q(w_{\sigma_j}, w'_{\sigma_j})  \prod_{i=1}^n\rd w_i\\
    &=\int_{\Sigma^n} F(w_{\sigma_1},\ldots,w_{\sigma_n}) \prod_{j=1}^k q(w_j,w'_j) \prod_{i=1}^n\rd w_i\\
    &=\sgn(\sigma) \int_{\Sigma^n} F(w_1,\ldots,w_n) \prod_{j=1}^{k} q(w_j, w'_{j})  \prod_{i=1}^n\rd w_i,
\end{split}
\end{equation}
where we relabeled the $w$-variables in the first step, and applied the antisymmetry of $F$ in the last step. This proves the second part of the lemma.
\end{proof}

Now we are ready to prove Lemma \ref{lm:hatD_version2}.

\begin{proof}[Proof of Lemma \ref{lm:hatD_version2}]

The first part of Lemma \ref{lm:symmetry} implies that we could replace a determinant $\det\left[ p_i(w_j)\right]_{i,j=1}^n$ by a single term $n!\prod_{i=1}^n p_i(w_i)$ in a $n$-fold integral with respect to $w_1,\ldots,w_n$ as long as the remaining part of the integrand is antisymmetric in these variables. We apply this lemma to the Cauchy determinant $\rC(U^{(1)};U^{(2)}\sqcup \hat V^{(1)})$ in \eqref{eq:hat_D_simplification} and use the fact that the remaining part of the integrand is antisymmetric in the variables $u_1^{(1)},\ldots,u_{n_1}^{(1)}$ due to the first Cauchy determinant $\rC(\hat V^{(m)}\sqcup \cdots \sqcup \hat V^{(1)};U^{(1)})=\rC(\hat V; U^{(1)})$. Thus, we could replace $\rC(U^{(1)};U^{(2)}\sqcup \hat V^{(1)})$ by
\begin{equation}
\label{eq:aux_03_proof_Prop}
n_1! \prod_{\ell_2=1}^{n_2} \frac{1}{u_{\ell_2}^{(1)} - u_{\ell_2}^{(2)}} \cdot \prod_{\hat\ell_1=1}^{k_1} \frac{1}{\hat u_{\hat \ell_1}^{(1)} - \hat v_{\hat \ell_1}^{(1)}}
\end{equation}
where 
\begin{equation}
\hat u_{\hat\ell_1}^{(1)} = u_{n_{2}+\hat\ell_1}^{(1)},\quad \hat\ell_1=1,\ldots,k_1=n_1-n_2.
\end{equation}
On the other hand, we also note that the remaining part in the integrand is also antisymmetric in $\hat v_{1}^{(1)},\ldots,\hat v_{k_1}^{(1)}$ due to the Cauchy determinant $\rC(\hat V^{(m)}\sqcup \cdots \sqcup \hat V^{(1)};U^{(1)})$. Thus, we apply the first part of Lemma \ref{lm:symmetry} again, and can replace $\prod_{\hat\ell_1=1}^{k_1} \frac{1}{\hat u_{\hat \ell_1}^{(1)} - \hat v_{\hat \ell_1}^{(1)}}$ in \eqref{eq:aux_03_proof_Prop} by the Cauchy determinant $(k_1!)^{-1}\rC(\hat U^{(1)}; \hat V^{(1)})$, where 
\begin{equation}
\hat U^{(1)}:=(\hat u_1^{(1)},\ldots,\hat u_{k_1}^{(1)}) = (u_{n_2+1}^{(1)},\ldots,u_{n_1}^{(1)}).
\end{equation}
Now we have the following expression for $\hat{\mathcal{D}}_{\bsn}$
\begin{equation}
\label{eq:aux_04_proof_Prop}
    \begin{split}
      \hat{\mathcal{D}}_{\bsn} & = \prod_{i=1}^{m-1}\frac{n_i!}{k_i!} \cdot \prod_{i=1}^m \prod_{\ell_i=1}^{n_i}   \int_{\Gamma_{i,\rL}^\inn} \frac{\rd u_{\ell_i}^{(i)}}{2\pi\rmi}  
      \cdot \prod_{i=1}^m \prod_{\hat\ell_i=1}^{k_i} \int_{\Gamma_{1,\rR}} \frac{\rd \hat v_{\ell_i}^{(i)}}{2\pi\rmi}   \\
      &\quad \cdot \frac{n_1!}{k_1!}\cdot \rC(\hat V;U^{(1)})\cdot \rC(\hat U^{(1)};\hat V^{(1)}) \cdot \prod_{i=2}^m \rC(U^{(i)}; U^{(i+1)}\sqcup \hat V^{(i)}) \cdot \prod_{\ell_2=1}^{n_2} \frac{1}{ u_{\ell_2}^{(1)} - u_{ \ell_2}^{(2)}} \cdot  \prod_{i=1}^m \frac{\prod_{\ell_i=1}^{n_i}F_i(u_{\ell_i}^{(i)})}{\prod_{\hat\ell_i=1}^{k_i} f_i(\hat v_{\hat\ell_i}^{(i)})}.
    \end{split}
    \end{equation}
Note that comparing to \eqref{eq:hat_D_simplification}, the formula above simply did the following replacement in the integrand
\begin{equation}
\rC(U^{(1)};U^{(2)}\sqcup \hat V^{(1)}) \to \frac{n_1!}{k_1!} \rC(\hat U^{(1)}; \hat V^{(1)}) \cdot \prod_{\ell_2=1}^{n_2} \frac{1}{ u_{\ell_2}^{(1)} - u_{ \ell_2}^{(2)}}.
\end{equation}
Now we consider the factor $\rC(U^{(2)};U^{(3)}\sqcup \hat V^{(2)})$ in \eqref{eq:aux_04_proof_Prop}. Note that the remaining part is still antisymmetric in $u_1^{(2)},\ldots,u_{n_2}^{(2)}$ by part (ii) of Lemma \ref{lm:symmetry}. Similarly to the previous factor, we have the following replacement in the integrand
\begin{equation}
\rC(U^{(2)};U^{(3)}\sqcup \hat V^{(2)}) \to \frac{n_2!}{k_2!} \rC(\hat U^{(2)}; \hat V^{(2)}) \cdot \prod_{\ell_3=1}^{n_3} \frac{1}{ u_{\ell_3}^{(2)} - u_{ \ell_3}^{(3)}},
\end{equation}
where
\begin{equation}
\hat U^{(2)} := (\hat u_1^{(2)},\ldots,\hat u_{k_2}^{(2)}) = (u_{n_3+1}^{(2)},\ldots,u_{n_2}^{(2)}).
\end{equation}
We repeat this procedure and finally arrive at
\begin{equation}
\label{eq:aux_05_proof_Prop}
    \begin{split}
      \hat{\mathcal{D}}_{\bsn} & = \left(\prod_{i=1}^{m-1}\frac{n_i!}{k_i!}\right)^2 \cdot \prod_{i=1}^m \prod_{\ell_i=1}^{n_i}   \int_{\Gamma_{i,\rL}^\inn} \frac{\rd u_{\ell_i}^{(i)}}{2\pi\rmi}  
      \cdot \prod_{i=1}^m \prod_{\hat\ell_i=1}^{k_i} \int_{\Gamma_{1,\rR}} \frac{\rd \hat v_{\ell_i}^{(i)}}{2\pi\rmi}\\
      &\quad \rC(\hat V;U^{(1)}) 
      \cdot \prod_{i=1}^m \rC(\hat U^{(i)};\hat V^{(i)})  \cdot \prod_{i=2}^m\prod_{\ell_i=1}^{n_i} \frac{1}{ u_{\ell_i}^{(i-1)} - u_{ \ell_i}^{(i)}} \cdot  \prod_{i=1}^m \frac{\prod_{\ell_i=1}^{n_i}F_i(u_{\ell_i}^{(i)})}{\prod_{\hat\ell_i=1}^{k_i} f_i(\hat v_{\hat\ell_i}^{(i)})},
    \end{split}
\end{equation}
where
\begin{equation}
\hat U^{(i)}:= (\hat u_1^{(i)},\ldots,\hat u_{k_i}^{(i)}) = (u_{n_{i+1}+1}^{(i)},\ldots,u_{n_i}^{(i)}),\quad 1\le i\le m.
\end{equation}

\bigskip

Now we absorb all the $u$-factors in the integrand of \eqref{eq:aux_05_proof_Prop} into the Cauchy determinants. Note that all the variables $\hat u_{\hat\ell_i}^{(i)}= u_{n_{i+1}+\hat \ell_i}^{(i)}$ for $1\le i\le m$ and $1\le \hat \ell_i \le k_i$ are distinct. Moreover, for each $1\le \ell \le n_1$, there exists a unique pair $(i,\hat\ell_i)$ such that $1\le i\le m$, $1\le \hat\ell_i\le k_i$, and $\ell = n_{i+1}+\hat\ell_i$. We group all the factors involving $u_\ell^{(j)}$ together for each fixed $\ell$, which give
\begin{equation}
\prod_{j=1}^{i-1}\frac{F_j(u_{\ell}^{(j)})}{u_{\ell}^{(j)}-u_{\ell}^{(j+1)}} \cdot F_i(u^{(i)}_\ell).
\end{equation}
Then we absorb this term to the $\hat\ell_i$-th row of the Cauchy determinant $\rC(\hat U^{(i)};\hat V^{(i)})$. Note that every $u$-factor is uniquely absorbed into one row in one Cauchy determinant. Thus, we can evaluate each $u_\ell^{(i)}$ integral inside the corresponding Cauchy determinant if $i>1$. The restriction of $i>1$ is because $u_\ell^{(1)}$ also appears in the first Cauchy determinant $\rC(\hat V^{(1)};U^{(1)})$. Now the Cauchy determinant $\rC(\hat U^{(i)};\hat V^{(i)})$ becomes a new determinant whose entry at the $\hat \ell_i$-th row and $\hat \ell'_i$-th column, $1\le \hat\ell_i,\hat\ell'_i\le k_i$, is given by
\begin{equation}
 \int_{\Gamma_{2,\rL}^\inn} \frac{\rd u_{\ell}^{(2)}}{2\pi\rmi} \cdots \int_{\Gamma_{i,\rL}^\inn} \frac{u^{(i)}_{\ell}}{2\pi\rmi}\prod_{j=1}^{i-1}\frac{F_j(u^{(j)}_{\ell})}{u^{(j)}_{\ell}-u^{(j+1)}_{\ell}} \cdot \frac{F_i( u^{(i)}_\ell)}{\hat u^{(i)}_{\hat\ell_i} -\hat v^{(i)}_{\hat \ell'_i}} = h_i\left(u_{n_{i+1}+\ell_i}^{(1)},\hat v_{\hat\ell'_i}^{(i)}\right),
\end{equation}
where $\ell = n_{i+1}+\hat\ell_i$. Note that $u_\ell^{(i)} =\hat u_{\hat \ell_i}^{(i)}$. 
The identity \eqref{eq:hatD_expression2} follows after relabeling the variables in $U^{(1)}$.
\end{proof}

\bigskip

At the end of this subsection, we prove Proposition \ref{prop:alt_formula}.

\begin{proof}[Proof of Proposition \ref{prop:alt_formula}]

We apply Lemma \ref{lm:Andreief_ext} for \eqref{eq:hatD_expression2} and get, when $k_i=n_i-n_{i+1}\ge 0$, $1\le i\le m$,
\begin{equation}
\hat{\mathcal{D}}_{\bsn} = \prod_{i=1}^m\frac{(n_i!)^2}{k_i!} \prod_{i=1}^m \prod_{\hat \ell_i=1}^{k_i} \int_{\Gamma_{1,\rL}} \frac{\rd \hat u_{\hat \ell_i}^{(i)}}{2\pi\rmi}
\det\left[ \int_{\Gamma_{1,\rR}} h_i\left(\hat u_{\hat \ell_i}^{(i)},v\right)\frac{1}{f_i(v)} \cdot \frac{1}{v- \hat u_{\hat\ell_j}^{(j)}} \frac{\rd v}{2\pi\rmi}\right]_{(i,\hat\ell_i),(j,\hat\ell_j)},
\end{equation}
where the row/column indices are chosen from $\{(i,\hat\ell_i): 1\le i\le m, 1\le \hat\ell_i\le k_i\}$. It is direct to check the entry in the determinant equals to $\bfK\left(i,\hat u_{\hat\ell_i}^{(i)}; j, \hat u_{\hat\ell_j}^{(j)} \right )$, see the definitions of $\bfK$ in \eqref{eq:def_bfK} and $h_i$ in \eqref{eq:def_hi}. Inserting the above formula in \eqref{eq:sum_hatD}, we obtain
\begin{equation}
\prob\left(\bigcap_{i=1}^m \left\{\Airy(\alpha_i)\le \beta_i \right\}\right) = \sum_{k_1,\ldots,k_m\ge 0} \frac{1}{k_1!\cdots k_m!} \prod_{\hat \ell_i=1}^{k_i} \int_{\Gamma_{1,\rL}} \frac{\rd \hat u_{\hat \ell_i}^{(i)}}{2\pi\rmi}
\det\left[\bfK\left(i,\hat u_{\hat\ell_i}^{(i)}; j, \hat u_{\hat\ell_j}^{(j)} \right )\right]_{(i,\hat\ell_i),(j,\hat\ell_j)}.
\end{equation}
This is the same as the Fredholm determinant expansion \eqref{eq:Fredholm_series2}. This completes the proof.
\end{proof}

\section{Equivalence of two Fredholm determinants}
\label{sec:Equivalence}

In this section, we further rewrite the Fredholm determinant formula in Proposition \ref{prop:alt_formula} and show the equivalence between this formula and \eqref{eq:def_AiryProcess}.

We first conjugate the kernel $\bfK$ and rewrite the Fredholm determinant $\det(\rI +\bfK)$ to a new Fredholm determinant of an operator defined on $L^2(\realR_+)$. Recall the definition of the kernel $\bfK$ in \eqref{eq:def_bfK}. It can be rewritten as 
\begin{equation}
\bfK(i,z;j,u) = \int_0^\infty L_1(i,z;\lambda) L_2(\lambda;j,u)\rd \lambda,
\end{equation}
where 
\begin{equation}
\label{eq:def_L1}
L_1(i,z;\lambda) 
= \begin{dcases}
    \int \frac{\rd v}{2\pi\rmi} \frac{f_1(z)}{f_1(v)e^{\lambda v}}\cdot \frac{1}{z-v}, & i=1,\\
    \int \frac{\rd v}{2\pi\rmi} \prod_{\ell=2}^i \int \frac{\rd u_\ell}{2\pi\rmi} \frac{f_1(z)}{f_i(v)e^{\lambda v}}
    \cdot 
    \frac{\prod_{\ell=2}^i F_\ell(u_\ell)}{(z-u_2)\cdot \prod_{\ell=2}^{i-1}(u_\ell -u_{\ell+1})\cdot (u_i -v)}, & 2\le i\le m,
  \end{dcases}
\end{equation}
and
\begin{equation}
L_2(\lambda;j,u) = e^{\lambda u}.
\end{equation}
The integration contours in \eqref{eq:def_L1} are the same as in the definition of $\bfK$ in \eqref{eq:def_bfK}, but here we will make a specific choice of the contour for notation convenience when we do the contour deformations later. We take the $u$ contour to be the same as $\Gamma_{1,\rL}$, and the $u_\ell$, $2\le \ell \le m$, contours to be $\Gamma_{\ell,\rL}^\inn$. We also set the $v$-contour to be $\Gamma_{1,\rR}$. We will also use the contours $\Gamma_{\ell,\rL}^\out$, $2\le \ell\le m$. All these contours are defined the same way as in Section \ref{sec:Liu_formula}. See Figure \ref{fig:Gamma_contours} for an illustration.

One can also view $L_1$ and $L_2$ as operators, and denote $\bfL=L_2L_1$ the product operator with the following kernel, after using the above notations of the contours,
\begin{equation}
\label{eq:def_L}
\begin{split}
    \bfL(\lambda,\theta) &= \sum_{i=1}^m \int_{\Gamma_{1,\rL}} \frac{\rd u_1}{2\pi\rmi} L_2(\lambda;i,u_1)L_1(i,u_1;\theta)\\
           &=\sum_{i=1}^m \int_{\Gamma_{1,\rL}} \frac{\rd u_1}{2\pi\rmi} \int_{\Gamma_{1,\rR}} \frac{\rd v}{2\pi\rmi} \prod_{\ell=2}^i \int_{\Gamma_{\ell,\rL}^\inn }\frac{\rd u_\ell}{2\pi\rmi}
           \cdot \frac{f_1(u_1)e^{\lambda u_1}}{f_i(v)e^{\theta v}} \frac{\prod_{\ell=2}^i F_\ell(u_\ell)}{\prod_{\ell=1}^{i-1}(u_\ell-u_{\ell+1})\cdot (u_i-v)}.
\end{split}
\end{equation}
Recall that $\det(\rI+\bfK)$ is well-defined by its series expansion which is absolutely convergent. Also note that if any of the two Fredholm determinants $\det(\rI+AB)$ and $\det(\rI+BA)$ has an absolutely convergent series expansion, then the identity $\det(\rI+AB) =\det(\rI+BA)$ holds and both have absolutely convergent series expansions.  Therefore $\det(\rI +\bfK) =\det(\rI + \bfL)$, where
\begin{equation}
\label{eq:bfL_expansion}
\det(\rI + \bfL) =\sum_{k=0}^\infty \frac{1}{k!} \int_0^\infty \cdots \int_0^\infty \det\left[ \bfL(\lambda_i,\lambda_j)\right]_{i,j=1}^k \prod_{i=1}^n\rd\lambda_i.
\end{equation}
Moreover, the summation is absolutely convergent.

\bigskip

It turns out that the kernel $\bfL$ can be further simplified.
\begin{lm}
\label{lm:kernel_decomposition}
We have
\begin{equation}
\begin{split}
    &\bfL(\lambda,\theta) \\
    &= \sum_{k=1}^m (-1)^k \sum_{1\le i_1<\cdots<i_k\le m} \int_0^\infty \cdots \int_0^\infty A_1(\lambda; i_1,\gamma_1) \cdot \prod_{\ell=1}^{k-1} B(i_\ell,\gamma_\ell;i_{\ell+1},\gamma_{\ell+1}) \cdot  A_2(i_k,\gamma_k;\theta) \prod_{\ell=1}^k\rd\gamma_\ell,
    \end{split}
\end{equation}
where
\begin{equation}
A_1(\lambda;i,\gamma) = \int_{\Gamma_{1,\rL}} f_i(u) e^{(\lambda+\gamma) u}\frac{\rd u}{2\pi\rmi},\quad  A_2(i,\gamma;\theta) = \int_{\Gamma_{1,\rR}} \frac{1}{f_i(v) e^{(\theta+\gamma) v}} \frac{\rd v}{2\pi\rmi},\quad 1\le i\le m,
\end{equation}
and
\begin{equation}
B(i,\gamma;i',\gamma') = \int_{\rmi\realR} e^{(\gamma'-\gamma)w}\cdot \prod_{\ell=i+1}^{i'}F_\ell(w) \frac{\rd w}{2\pi\rmi},\quad 1\le i< i'\le m.
\end{equation}
\end{lm}

\begin{proof}[Proof of Lemma \ref{lm:kernel_decomposition}]
Recall the definition of $\bfL(\lambda,\theta)$ in \eqref{eq:def_L}. We deform the contours $\Gamma_{\ell,\rL}^\inn$ to $\Gamma_{\ell,\rL}^{\out}$, $\ell=2,\ldots,i$, subsequently. Note that the $u_\ell$-integral will result in two terms: one term gives the residue of the integrand when $u_\ell=u_{\ell-1}$, and the other term gives the integral along the new contour $\Gamma_{\ell,\rL}^\out$. By sorting the terms, we get
\begin{equation}
\begin{split}
&\sum_{i=1}^m \int_{\Gamma_{1,\rL}} \frac{\rd u_1}{2\pi\rmi} \int_{\Gamma_{1,\rR}} \frac{\rd v}{2\pi\rmi} \prod_{\ell=2}^i \int_{\Gamma_{\ell,\rL}^\inn }\frac{\rd u_\ell}{2\pi\rmi}
           \cdot \frac{f_1(u_1)e^{\lambda u_1}}{f_i(v)e^{\theta v}} \frac{\prod_{\ell=2}^i F_\ell(u_\ell)}{\prod_{\ell=1}^{i-1}(u_\ell-u_{\ell+1})\cdot (u_i-v)}\\
&=\sum_{i=1}^m\sum_{k=1}^i  \sum_{1\le i_1<\cdots <i_{k}=  i} \int_{\Gamma_{i_1,\rL}^{\out}} \frac{\rd u_{i_1}}{2\pi\rmi} \cdots \int_{\Gamma_{i_k,\rL}^\out} \frac{\rd u_{i_k}}{2\pi\rmi} \int_{\Gamma_{1,\rR}} \frac{\rd v}{2\pi\rmi} \frac{f_{i_1}(u_{i_1})e^{\lambda u_{i_1}}}{f_i(v)e^{\theta v}} \frac{\prod_{\ell=2}^{k}\prod_{j=i_{\ell-1}+1}^{i_{\ell}}F_{j}(u_{i_{\ell}}) }{\prod_{\ell=1}^{k-1}(u_{i_\ell} -u_{i_{\ell+1}})\cdot (u_{i_k}-v)}\\
&=\sum_{k=1}^m  \sum_{1\le i_1<\cdots <i_{k}\le m} \int_{\Gamma_{i_1,\rL}^{\out}} \frac{\rd u_{i_1}}{2\pi\rmi} \cdots \int_{\Gamma_{i_k,\rL}^\out} \frac{\rd u_{i_k}}{2\pi\rmi} \int_{\Gamma_{1,\rR}} \frac{\rd v}{2\pi\rmi} \frac{f_{i_1}(u_{i_1})e^{\lambda u_{i_1}}}{f_{i_k}(v)e^{\theta v}} \frac{\prod_{\ell=2}^{k}\prod_{j=i_{\ell-1}+1}^{i_{\ell}}F_{j}(u_{i_{\ell}}) }{\prod_{\ell=1}^{k-1}(u_{i_\ell} -u_{i_{\ell+1}})\cdot (u_{i_k}-v)},
\end{split}
\end{equation}
where we set $\Gamma_{1,\rL}^\out=\Gamma_{1,\rL}$ for notation convenience.
Recall the definition of the functions $F_i$ in \eqref{eq:def_F}. The functions $\prod_{j=i+1}^{i'}F_j(w) =e^{(\alpha_{i'}-\alpha_i)w^2 + (\beta_{i'}-\beta_i)}$ decay super-exponentially fast when $w\to\infty$ along any vertical line since $\alpha_{i'}-\alpha_i>0$. Thus, we could deform the $\Gamma_{i_\ell,\rL}^\out$, $2\le\ell\le k$, to vertical lines as long as these contours are of the same order, and lie between the two contours $\Gamma_{i_1,\rL}^\out$ and $\Gamma_{1,\rR}$. See Figure \ref{fig:deformation_contours} for an illustration.

\begin{figure}
    \centering
    \begin{tikzpicture}[scale=1.5, decoration={
  markings,
  mark=at position 0.75 with {\arrow{>}}}
  ] 
  \draw[-latex] (-3,0) -- (3,0) node[right] {$\Re$};
  \draw[-latex] (0,-2) -- (0,2) node[above] {$\Im$};

  \draw[postaction={decorate},thick,blue] (-2,-1.5) -- (-2,1.5);
  \draw[postaction={decorate},thick,blue] (-0.5,-1.5) -- (-0.5,1.5);

  \draw[postaction={decorate},blue,thick,dashed] plot[smooth] coordinates {(-3,-1) (-2.5,-0.5) (-2,0) (-2.5,0.5) (-3,1)};
  \draw[postaction={decorate},blue,thick,dashed] plot[smooth] coordinates {(-1.5,-1) (-1,-0.5) (-0.5,0) (-1,0.5) (-1.5,1)};
  \draw[postaction={decorate},blue,thick] plot[smooth] coordinates {(-3.5,-1) (-3,-0.5) (-2.5,0) (-3,0.5) (-3.5,1)};

  \draw[postaction={decorate},red,thick] plot[smooth] coordinates {(3,-1) (2,-0.5) (1,0) (2,0.5) (3,1)};

  \filldraw[blue] (-1.6,0.2) circle (0.5pt);
 \filldraw[blue] (-1.5,0.2) circle (0.5pt);
  \filldraw[blue] (-1.4,0.2) circle (0.5pt);
  \node at (-3.5,0.5) {$\Gamma_{i_1,\rL}^\out$};
  \node at (-3,1.2) {$\Gamma_{i_2,\rL}^\out$};
  \node at (-1.6,1.2) {$\Gamma_{i_k,\rL}^\out$};
  \node at (2.7,1.2) {$\Gamma_{1,\rR}$};
  \node at (-2,1.8) {$u_{i_2}$};
  \node at (-0.5,1.8) {$u_{i_k}$};
\end{tikzpicture}
    \caption{Deformation of the contours: Initially, for each $2\le i\le k$, the contour of $u_{i_\ell}$ is the dashed contour $\Gamma_{i_k,\rL}^\out$; It is deformed to a solid vertical line. After the deformation, $\Re(u_{i_1}) < \Re (u_{i_2}) <\cdots <\Re (u_{i_k}) <\Re(v)$ when $u_{i_1}\in\Gamma_{i_1,\rL}^\out$, $v\in\Gamma_{1,\rR}$, and $u_{i_\ell}$ from the deform contour, $2\le \ell \le k$.}
    \label{fig:deformation_contours}
\end{figure}
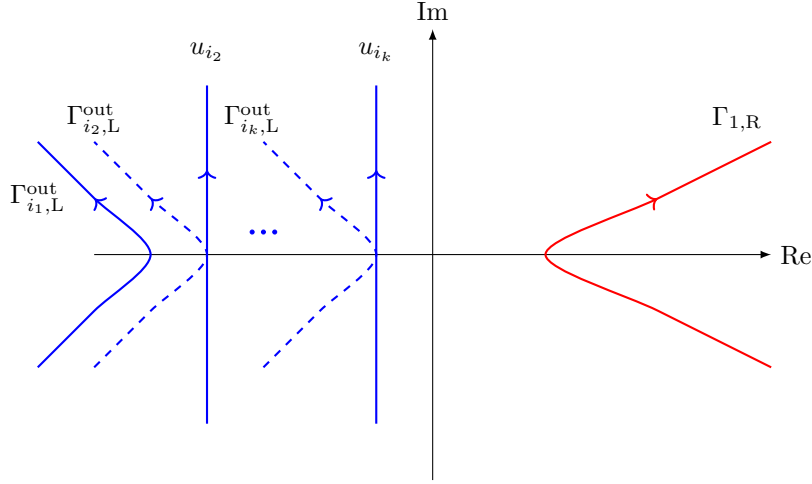

With the new contours, $\Re(u_{i_\ell}-u_{i_{\ell+1}})<0$ for all $1\le \ell\le k-1$, and $\Re(u_{i_k}-v)<0$. Hence we write
\begin{equation}
\frac{1}{u_{i_{\ell}} -u_{i_{\ell+1}}} = -\int_0^\infty e^{\gamma_\ell(u_{i_\ell}-u_{i_{\ell+1}})} \rd \gamma_{\ell}, \qquad \frac{1}{u_{i_k}-v} =-\int_0^\infty e^{\gamma_{k}(u_{i_k}-v)}\rd\gamma_{k}.
\end{equation}
The lemma follows immediately.
\end{proof}

With Lemma \ref{lm:kernel_decomposition}, we are able to express  $\bfL$ as the product of three operators. We extend the definition of $B$ to
\begin{equation}
\label{eq:def_B}
B(i,\gamma;i',\gamma') = \begin{dcases}
     \int_{\rmi\realR} e^{(\gamma'-\gamma)w}\cdot \prod_{\ell=i+1}^{i'}F_\ell(w) \frac{\rd w}{2\pi\rmi}, & 1\le i< i'\le m, \text{ and }\gamma,\gamma'>0, \\
     0,&\text{elsewhere}.
\end{dcases}
\end{equation}
Then 
\begin{equation}
\bfL = - A_1\left(\sum_{k=0}^{m-1} (-1)^k B^{k}\right) A_2.
\end{equation}
Similarly to the arguments before \eqref{eq:bfL_expansion}, we can write
\begin{equation}
\det (\rI +\bfL ) = \det \left(\rI  - \left(\sum_{k=0}^{m-1} (-1)^k B^{k}\right) A\right),
\end{equation}
where $A=A_2A_1$ with kernel given by
\begin{equation}
\label{eq:def_A}
A(i,\lambda;j,\theta) = \int_0^\infty \int_{\Gamma_{1,\rL}}\frac{\rd u}{2\pi\rmi} \int_{\Gamma_{1,\rR}} \frac{\rd v}{2\pi\rmi}\frac{f_j(u)e^{(\theta+\gamma)u}}{f_i(v)e^{(\lambda+\gamma)v}}\rd \gamma.
\end{equation}

Now we use the fact that $B$ is strictly upper-triangle (with zero diagonal entries), hence multiplying the operator $\rI+B$ inside the Fredholm determinant will not affect its value. This implies
\begin{equation}
\label{eq:final_bfL}
\det(\rI +\bfL) =\det\left((\rI+B)\left(\rI  - \left(\sum_{k=0}^{m-1} (-1)^k B^{k}\right) A\right)\right) =\det (\rI +B-A),
\end{equation}
where we used the fact that $B^m=0$ by its definition.

\bigskip
The last piece of the proof we need is the following lemma.

\begin{lm}
\label{lm:kernel_evaluation}
We have the following identities.
\begin{enumerate}[(i)]
\item For all $1\le i,j\le m$, and $\lambda,\theta>0$,
\begin{equation}
\label{eq:evaluation_A}
\begin{split}
&A(i,\lambda;j,\theta)\\
&= e^{-\frac23\alpha_i^3 -(\lambda+\beta_i)\alpha_i+ \frac{2}{3}\alpha_j^3 +(\theta+\beta_j)\alpha_j }\int_0^\infty e^{- (\alpha_i-\alpha_j)\gamma}\Ai\left(\lambda+\beta_i+\alpha_i^2 +\gamma\right)\Ai\left(\theta+\beta_j+\alpha_j^2+\gamma\right) \rd\gamma.
\end{split}
\end{equation}
\item For all $1\le i< j\le m$,
\begin{equation}
\label{eq:evaluation_B}
\begin{split}
    &B(i,\lambda;j,\theta) \\
    &= \frac{1}{2\sqrt{\pi(\alpha_j-\alpha_i)}} e^{-\frac{1}{4(\alpha_j-\alpha_i)}(\beta_j+\theta -\beta_i-\lambda)^2}\\
    &=e^{-\frac23\alpha_i^3 -(\lambda+\beta_i)\alpha_i+ \frac{2}{3}\alpha_j^3 +(\theta+\beta_j)\alpha_j }\int_{-\infty}^\infty e^{- (\alpha_i-\alpha_j)\gamma}\Ai\left(\lambda+\beta_i+\alpha_i^2 +\gamma\right)\Ai\left(\theta+\beta_j+\alpha_j^2+\gamma\right) \rd\gamma.
\end{split}
\end{equation}
\end{enumerate}
\end{lm}

\begin{proof}[Proof of Lemma \ref{lm:kernel_evaluation}]
Recall the definitions of $f_i(w) = e^{-\frac13w^3 +\alpha_iw^2 +\beta_i w}$ in \eqref{eq:def_f}, and the Airy function in \eqref{eq:def_Airy}. Inserting these in \eqref{eq:def_A}, we obtain \eqref{eq:evaluation_A}  by a direct computation.

Similarly, recall the definition of $F_\ell$ function in \eqref{eq:def_F}. The first equation of \eqref{eq:evaluation_B} follows from a direct computation of \eqref{eq:def_B}. The second part follows from the following identity (see \cite[Lemma 2.6]{Okounkov02})
\begin{equation}
\int_{-\infty}^\infty e^{xz}\Ai(z+a)\Ai(z+b)\rd z =\frac{1}{2\sqrt{\pi x}}e^{\frac{x^3}{12}-\frac{a+b}{2}x -\frac{(a-b)^2}{4x}},\quad x>0.
\end{equation}
\end{proof}

Now we insert this lemma to \eqref{eq:final_bfL}. The Fredholm determinant equals to \eqref{eq:def_AiryProcess} after a conjugation and a simple notation change.

\section*{Acknowledgements}

Both authors are supported in part by the NSF grants DMS-2246683 and DMS-2054735.

\bibliographystyle{alpha}
\def\cydot{\leavevmode\raise.4ex\hbox{.}}

\end{document}